\documentclass[a4paper,12pt]{article}
\usepackage{graphicx}
\usepackage[intlimits]{amsmath}
\usepackage{amsxtra, amsthm, amssymb, latexsym}
\usepackage[english]{babel}
\usepackage{enumerate}
\numberwithin{equation}{section}
\numberwithin{figure}{section}

\usepackage{tikz}
\usepackage{color}

\begin{document}
\sloppy 
\newtheorem{theorem}{Theorem}[section]
\newtheorem{lemma}[theorem]{Lemma}
\newtheorem{corollary}[theorem]{Corollary}

\theoremstyle{definition} 
\newtheorem{definition}[theorem]{Definition}
\newtheorem{remark}[theorem]{Remark}
\newtheorem{example}[theorem]{Example}
\newtheorem{assumption}[theorem]{Assumption}

\newcommand{\mathbbm}[1]{{#1\!\!#1}}
\newcommand{\ind}{{\mathbbm{1}}}

\newcommand{\CM}{\mathcal{M}}
\newcommand{\CMc}{\overline{\CM}}
\newcommand{\smfrac}[2]{\mbox{$\frac{#1}{#2}$}}
\newcommand{\geqs}{\geqslant}
\newcommand{\leqs}{\leqslant}
\newcommand{\nn}{{(n)}}
\newcommand{\mm}{{(m)}}

\newcommand{\Lip}{{\mathrm{L}}}
\newcommand{\BL}{{\mathrm{BL}}}
\newcommand{\TV}{{\mathrm{TV}}}
\newcommand{\FM}{{\mathrm{FM}}}
\newcommand{\pair}[2]{\left\langle #1 , #2 \right\rangle}
\newcommand{\Int}[4]{\int_{#1}^{#2}\! #3 \, #4}
\newcommand{\map}[3]{#1 : #2 \rightarrow #3}
\def\R{\mathbb{R}}
\def\Rp{\mathbb{R}^+}
\def\N{\mathbb{N}}
\def\Np{\mathbb{N}^+}
\def\eps{\varepsilon}
\def\aeps{a}
\def\K{\mathcal{K}}

\title{Mild solutions are weak solutions\\in a class of (non)linear measure-valued\\evolution equations on a bounded domain}
\author{Joep H.M.~Evers\thanks{Department of Mathematics, Simon Fraser University, Burnaby, Canada, and Department of Mathematics and Statistics, Dalhousie University, Halifax, Canada. Email: \texttt{jevers@sfu.ca}.}}
\date{}
\maketitle



\begin{abstract}
We study the connection between mild and weak solutions for a class of measure-valued evolution equations on the bounded domain $[0,1]$. Mass moves, driven by a velocity field that is either a function of the spatial variable only, $v=v(x)$, or depends on the solution $\mu$ itself: $v=v[\mu](x)$. The flow is stopped at the boundaries of $[0,1]$, while mass is gated away by a certain right-hand side. In previous works \cite{EvHiMu_JDE,EvHiMu_SIMA}, we showed the existence and uniqueness of appropriately defined mild solutions for $v=v(x)$ and $v=v[\mu](x)$, respectively. In the current paper we define weak solutions (by specifying the weak formulation and the space of test functions). The main result is that the aforementioned mild solutions are weak solutions, both when $v=v(x)$ and when $v=v[\mu](x)$.

%
\end{abstract}


\section{Introduction}
Measure-valued evolution equations have been used in a large number of recent mathematical publications to model for instance animal aggregations \cite{CanCarRos,CarrilloFornasierRosadoToscani}, structured populations \cite{Ackleh1,Gwiazda2,DG, Gwiazda1}, pedestrian dynamics \cite{CrisPiccTosinBook}, and defects in metallic crystals \cite{vMeurs14}. The majority of works that study well-posedness of measure-valued equations and properties of their solutions treat these equations in the full space --see for instance also \cite{BenzoniColomboGwiazda, CarrDiFranFigLauSlep11,Colombo_Crippa_Spirito,CrippaLecureux, TosinFrasca}-- although many relevant problems involve \textit{boundaries} and \textit{bounded domains}. Examples of such problems --apart from the ones mentioned above-- are intracellular transport processes, cf.~\cite[Section 1]{EversMBE}, and manufacturing chains \cite{Schleper}. Defining mathematically and physically `correct' boundary conditions is a challenge, however. The present paper is a continuation of the author's work (in collaboration with Hille and Muntean) that focuses explicitly on bounded domains and boundary conditions.\\
\\
Our first step is to consider a one-dimensional measure-valued transport equation restricted to the unit interval $[0,1]$ where mass moves according to a prescribed velocity field $v$ and stops when reaching the boundary. A short-hand notation for this equation is:
\begin{equation}\label{eq:main equation}
\frac{\partial}{\partial t}\mu_t + \frac{\partial}{\partial x} (v\,\mu_t) = F_f( \mu_t).
\end{equation}
Here, the perturbation map $\map{F_f}{\CM([0,1])}{\CM([0,1])}$ is given by $F_f(\mu):= f\cdot \mu$.\\
In \cite{EvHiMu_JDE}, we proved the well-posedness of this equation, in the sense of \emph{mild solutions}, and the convergence of solutions corresponding to a sequence $(f_n)_{n\in\N}$ in the right-hand side. Some specific choices for $(f_n)_{n\in\N}$ represent for instance effects in a boundary layer that approximate, as $n\to\infty$, sink or source effects localized on the boundary (flux boundary conditions). The boundary layer corresponds to exactly those regions in $[0,1]$ where the functions $f_n$ are nonzero.\\
\\
Next, we want to consider \eqref{eq:main equation} for velocity fields that are no longer \textit{fixed} elements of $\BL([0,1])$. Instead of $v$, we write $v[\mu]$ for the velocity field that depends functionally on the measure $\mu$. An example is 
\begin{equation}\label{eqn: example meas-dep v}
v[\mu](x) := \Int{[0,1]}{}{\K(x-y)}{d\mu(y)} = (\K*\mu)(x),
\end{equation}
where the convolution encodes nonlocal interactions due to a kernel $\K$ in a population with distribution $\mu$. This is a widely used choice of $v$, e.g.~in interacting particle systems or biological aggregation models. The  example \eqref{eqn: example meas-dep v} is a special case of the class of velocity fields (see Assumption \ref{ass: v properties}) that are admissible in the framework of the current paper.\\
\\
For such solution-dependent $v=v[\mu]$, the transport equation on $[0,1]$ becomes
\begin{equation}\label{eqn:main equation nonlin}
\frac{\partial}{\partial t}\mu_t + \frac{\partial}{\partial x} (v[\mu_t]\,\mu_t ) = F_f(\mu_t).
\end{equation}
The well-posedness of \eqref{eqn:main equation nonlin} for $f\in\BL([0,1])$ was proved in \cite{EvHiMu_SIMA} in the sense of mild solutions. The analysis turns out to build on the analysis for \eqref{eq:main equation}, hence it was useful to consider \eqref{eq:main equation} before the more general \eqref{eqn:main equation nonlin}. For $v=v[\mu]$, mild solutions are defined as the limit of a sequence $(\mu^k)_{k\in\N}$ of so-called \textit{Euler approximations}. Such $\mu^k$ is constructed on each subinterval $(t^k_j,t^k_{j+1}]$ as a mild solution to \eqref{eq:main equation} for velocity $v[\mu^k_{t^k_j}]$. Within a subinterval, $v[\mu^k_{t^k_j}]$ is a \textit{fixed} element of $\BL([0,1])$ that is the same for all time $t\in(t^k_j,t^k_{j+1}]$. For further details on how $\mu^k$ is defined, see Section \ref{sec: gen meas-dep v}.\\
\\
The works \cite{Canizo} and \cite{Gwiazda1} treat comparable models, however posed on \textit{infinite} domains. They consider weak solutions, and in fact, they construct those weak solutions that --roughly speaking-- correspond to `our' mild solutions.\\
\\
The aim of the current paper is:
\begin{center}
\textbf{to investigate how and in which sense the mild solutions from \cite{EvHiMu_JDE,EvHiMu_SIMA}\\correspond to weak solutions like the ones in e.g.~\cite{Canizo,Gwiazda1}.}
\end{center}
Note that the essence of weak solutions lies in the specific choice of the weak formulation and of the space of test functions that appear in the definition; cf.~e.g.~Definition~\ref{def:weak soln v prescr} with weak formulation \eqref{eqn:weak form v prescr} and space of test functions \eqref{eqn:weak v=v(x) test funct}.\\
\\
Compared to e.g.~\cite{Canizo,Gwiazda1}, our case is more complicated due to the bounded domain; the material flow is induced by the velocity $v$ in the interior, but it is stopped once characteristics reach any of the boundary points. The \textit{stopped flow} introduces subtleties when trying to find the appropriate definition of weak solutions. The domain in \cite{Gwiazda1} is in fact $[0,\infty)$. Their velocity is required to point inward at $x=0$, though, which is sufficient to make sure that no mass escapes the domain. For us, a demand on the sign of the velocity at $x=0$ or $x=1$ is too restrictive; cf.~the remark we make about this in \cite[Section 1]{EvHiMu_SIMA}.\\
\\
In the current work, we overcome these difficulties and give the appropriate definition of weak solutions. We believe this indeed is the appropriate definition because of the following main result of this paper, that consists of two parts. Formulated in plain words in a pseudo-theorem, the first part of this result reads:\\
\\
\textbf{Theorem.} \emph{Mild solutions to \eqref{eq:main equation} are weak solutions (in an appropriate sense).}\\
\\
A more precise formulation follows in Theorem~\ref{thm:mild is weak v=v(x)}. Next, we use this property on each of the subintervals in an Euler approximation and show that in the limit as the mesh size goes to zero, we obtain a weak solution to \eqref{eqn:main equation nonlin}. In other words:\\
\\
\textbf{Theorem.} \emph{Mild solutions to \eqref{eqn:main equation nonlin} are weak solutions.}\\
\\
This result is stated in full detail in Theorem \ref{thm:mild is weak v=v(mu)}.\\
\\
Our justification for speaking about \textit{appropriate} definition, is exactly the fact that we show the relation between these weak solutions and mild solutions in this paper (more about this in Section \ref{sec: gen meas-dep v}, directly after Definition \ref{def: mild soln nonlin}). Mild solutions have a considerable advantage over weak solutions in the sense that it is directly clear how they should be interpreted, whereas defining weak solutions involves some seemingly arbitrary choices. Which choices to make is not directly evident from modelling considerations.\\
On the other hand, as was argued in \cite[Section 1]{EvHiMu_SIMA}, the mild formulation in terms of the variation of constants formula \eqref{eq:VOC} follows directly from a probabilistic interpretation. For more details, see \cite[Section 6]{EvHiMu_JDE}. Moreover, the exact form of the variation of contsants formula is unambiguous, provided the system that is to be modelled.\\
Subsequently, mild solutions for $v=v[\mu]$ follow in a straight-forward manner, using the variation of constants formula as a building block; cf.~\eqref{eqn: Euler scheme} and Definition \ref{eqn: metric sup BL norm}.\\
\\
The structure of this paper is as follows. Section \ref{sec: summ tech} provides preliminaries on the stopped flow on the interval $[0,1]$ induced by the velocity field $v:[0,1]\to\R$. In Section \ref{sec:mild weak v prescr}, we recall the results from \cite{EvHiMu_JDE} regarding the existence and uniqueness of mild solutions to \eqref{eq:main equation}. We introduce the concept of weak solutions, and show in Theorem \ref{thm:mild is weak v=v(x)} that the mild solutions from \cite{EvHiMu_JDE} are weak solutions. Section \ref{sec: gen meas-dep v} briefly recalls the main ideas from \cite{EvHiMu_SIMA}: the construction of Euler approximations using solutions to the variation of constants formula as building blocks. Theorem \ref{thm: exist uniq nonlin} repeats the result that Euler approximations converge as the mesh size goes to zero; this result is an alternative way of saying that mild solutions to \eqref{eqn:main equation nonlin} exist and are unique. We show in Section \ref{sec: gen meas-dep v} that these mild solutions are weak solutions (in an appropriate sense). The paper is concluded by a section (Section \ref{sec: discussion}) in which we discuss the wider context of our results and the open issues that are subject for follow-up work.

%

\section{Preliminaries}\label{sec: summ tech} \label{sec: prop stopped flow}
This section contains the preliminaries that are needed for the arguments in this paper. These preliminaries were presented before in \cite{EvHiMu_JDE, EvHiMu_SIMA}. We assume that the reader is familiar with elementary measure-theoretical concepts, such as finite Borel measures, the total variation norm $\|\cdot\|_\TV$, and the dual bounded Lipschitz norm $\|\cdot\|_\BL^*$. An overview of the basic concepts used in this paper can be found in Appendix \ref{sec: basics meas th}.\\
\\
The rest of this section is devoted to properties of the flow induced on $[0,1]$ by some fixed $v\in\BL([0,1])$, a bounded Lipschitz velocity field. This flow is a fundamental mechanism in the model considered in this paper.\\
\\
We assume that a single particle (`individual') is moving in the domain $[0,1]$ deterministically, described by the differential equation for its position $x(t)$ at time $t$:
\begin{equation}\label{eq:indiv flow}
\left\{
  \begin{array}{l}
    \dot{x}(t)=v(x(t)), \\
    x(0)=x_0.
  \end{array}
\right.
\end{equation}
A solution to \eqref{eq:indiv flow} is unique, it exists for time up to reaching the boundary $0$ or $1$ and depends continuously on initial conditions. Let $x(\,\cdot\,;x_0)$ be this solution and $I_{x_0}$ be its \textit{maximal} interval of existence. Define
\[
\tau_\partial(x_0) := \sup I_{x_0} \in [0,\infty],
\]
i.e.~$\tau_\partial(x_0)$ is the time at which the solution starting at $x_0$ reaches the boundary (if it happens) when $x_0$ is an interior point. Note that $\tau_\partial(x_0)=0$ when $x_0$ is a boundary point where $v$ points outwards, while $\tau_\partial(x_0)>0$ when $x_0$ is a boundary point where $v$ vanishes or points inwards.\\
\\
The {\em stopped flow} on $[0,1]$ associated to $v$ is the family of maps $\Phi_t:[0,1]\to[0,1]$, $t\geqs 0$,  defined by
\begin{equation}\label{individualistic flow Phi}
\Phi_t(x_0) := \begin{cases} x(t;x_0),& \quad \mbox{if}\ t\in I_{x_0},\\
x(\tau_\partial(x_0);x_0), & \quad \mbox{otherwise}.
\end{cases}
\end{equation}
To lift the dynamics to the space of measures, we define $\map{P_t}{\CM([0,1])}{\CM([0,1])}$ by means of the push-forward under $\Phi_t$: for all $\mu\in\mathcal{M}([0,1])$,
\begin{equation}\label{Def Pt push forward}
P_t\mu := \Phi_t \# \mu = \mu\circ\Phi_t^{-1};
\end{equation}
see \eqref{eqn: def push-forw chapter bc}. Clearly, $P_t$ maps positive measures to positive measures and $P_t$ is mass preserving on positive measures. Since the family of maps $(\Phi_t)_{t\geqs 0}$ forms a semigroup, so do the maps $P_t$ in the space $\CM([0,1])$. That is, $(P_t)_{t\geqs0}$ is a {\em Markov semigroup} on $\CM[0,1]$ (cf.~\cite{Lasota-Myjak-Szarek}). The basic estimate
\begin{equation}\label{eqn:TV norm Pt}
\|P_t\mu\|_\TV \leqs \|\mu\|_\TV
\end{equation}
holds for $\mu\in\CM([0,1])$. In \cite[Section 2.2]{EvHiMu_SIMA} a number of other properties (bounds and Lipschitz-like estimates) of $(P_t)_{t\geqs0}$ are given.

\section{Mild and weak solutions for prescribed velocity}\label{sec:mild weak v prescr}
Mild solutions to \eqref{eq:main equation} are defined in the following sense:

\begin{definition}[See {\cite[Definition 2.4]{EvHiMu_JDE}}]\label{def: mild soln fixed v}
A {\em measure-valued mild solution} to the Cauchy-problem associated to \eqref{eq:main equation} on $[0,T]$ with initial value $\nu\in\CM([0,1])$ is a continuous map $\mu:[0,T]\to\CM([0,1])_\BL$ that is $\|\cdot\|_\TV$-bounded and that satisfies the variation of constants formula
\begin{equation}\label{eq:VOC}
\mu_t = P_t\,\nu \ +\ \int_0^t P_{t-s}F_f(\mu_s)\, ds\qquad \mbox{for all}\ t\in [0,T].
\end{equation}
\end{definition}
\noindent Amongst others, we showed in \cite{EvHiMu_JDE} that mild solutions in the sense of Definition~\ref{def: mild soln fixed v} exist and are unique. We repeat those results in the following theorem.

\begin{theorem}[Existence and uniqueness of mild solutions to \eqref{eq:main equation}]\label{thm:exist uniq v prescribed}
Let $f:[0,1]\to\R$ be a piecewise bounded Lipschitz function such that $v(x)\neq 0$ at any point $x$ of discontinuity of $f$. Then for each $T\geqs 0$ and $\mu_0\in\CM([0,1])$ there exists a unique continuous and locally $\|\cdot\|_\TV$-bounded solution to \eqref{eq:VOC}.
\end{theorem}

\begin{proof}
See {\cite[Propositions 3.1 and 3.3]{EvHiMu_JDE}} for details.
\end{proof}
\noindent In the rest of this section we will compare mild solutions provided by Theorem \ref{thm:exist uniq v prescribed} to solutions in a different sense: weak solutions. Recall that $\langle\cdot,\cdot\rangle$ denotes the duality paring between finite Borel measures on $[0,1]$ and bounded measurable functions on $[0,1]$; see \eqref{pairing}.

\begin{definition}[Weak solution to \eqref{eq:main equation}]\label{def:weak soln v prescr}
Fix $T\geqs0$, let $v\in\BL([0,1])$ and let $\map{f}{[0,1]}{\R}$ be piecewise bounded Lipschitz. Then $\map{\mu}{[0,T]}{\CM([0,1])}$ is a \emph{weak solution} to \eqref{eq:main equation} corresponding to initial condition $\nu_0\in\CM([0,1])$, if 
\begin{equation}\label{eqn:weak form v prescr}
\pair{\mu_T}{\psi(\cdot,T)}-\pair{\nu_0}{\psi(\cdot,0)} = \int_0^T \pair{\mu_t}{\partial_t\psi(\cdot,t) + \partial_x\psi(\cdot,t)\cdot v} \, dt + \int_0^T \pair{F_f(\mu_t)}{\psi(\cdot,t)}\,dt
\end{equation}
is satisfied for all 
\begin{equation}\label{eqn:weak v=v(x) test funct}
\psi\in \Lambda^T:=\bigg\{\psi\in C^1([0,1]\times[0,T])\,:\, \partial_x\psi(0,t)=\partial_x\psi(1,t) =0 \text{ for all } 0\leqs t \leqs T \bigg\}. 
\end{equation}
\end{definition}
\begin{remark}
Note that the boundary conditions imposed on the test functions are closely related to the behaviour of the stopped flow at the boundary points $x=0$ and $x=1$: no flux. This is a general phenomenon. See for instance \cite[pp.~63--64 and 140]{Taira}, where several spaces of test functions are given, depending on which behaviour at the boundary is to be modelled. For the sake of completeness, we note that \cite{Taira} treats Brownian motion (diffusion). The used operator $A$ is the corresponding infinitesimal generator acting --as is common in probabilistic literature-- \textit{on the test functions}, not on the solution itself.
\end{remark}
\noindent The main result of this section is the following theorem.
\begin{theorem}[Mild solutions to \eqref{eq:main equation} are weak solutions]\label{thm:mild is weak v=v(x)}
Let $\map{\mu}{[0,T]}{\CM([0,1])}$ be the mild solution provided by Theorem \ref{thm:exist uniq v prescribed}, corresponding to initial value $\nu_0\in\CM([0,1])$. Then, $\mu$ is a weak solution of \eqref{eq:main equation}.
\end{theorem}
\noindent For the proof of Theorem \ref{thm:mild is weak v=v(x)} we were inspired by the proof of \cite[Proposition 3.7]{Hoogwater-thesis}.
\begin{proof}

Let $\psi$ be an arbitrary element from the set of test functions given in \eqref{eqn:weak v=v(x) test funct}. Recall that $I_y$ is the maximal interval of existence of a solution to \eqref{eq:indiv flow} --i.e. restricted to $[0,1]$-- with initial condition $y$. Recall moreover that $\tau_\partial(y) = \sup I_{y}$, i.e.~$\tau_\partial(y)$ is the time at which the solution starting at $y$ reaches the boundary (if it happens) when $y$ is an interior point. Note that $\tau_\partial(y)=0$ when $y$ is a boundary point where $v$ points outwards, while $\tau_\partial(y)>0$ when $y$ is a boundary point where $v$ vanishes or points inwards.\\ Consider
\begin{align}\label{eqn:split before after tau}
\nonumber &\int_0^T \pair{P_t\,\nu_0}{\partial_x\psi(\cdot,t)\cdot v + \partial_t\psi(\cdot,t)}\,dt \\
\nonumber &\hspace{0.15\linewidth}= \int_0^T \pair{\nu_0}{\partial_x\psi(\Phi_t(\cdot),t)\cdot v(\Phi_t(\cdot)) + \partial_t\psi(\Phi_t(\cdot),t)}\,dt\\
\nonumber &\hspace{0.15\linewidth}= \int_{[0,1]} \int_0^T \bigg(\partial_x\psi(\Phi_t(y),t)\cdot v(\Phi_t(y)) + \partial_t\psi(\Phi_t(y),t)\bigg)\,dt\,d\nu_0(y)\\
\nonumber &\hspace{0.15\linewidth}= \int_{[0,1]} \int_0^{\tau_\partial(y)\wedge T} \dfrac{d}{dt}\psi(\Phi_t(y),t)\,dt\,d\nu_0(y) \\
&\hspace{0.15\linewidth}\,\,\,\,+ \int_{[0,1]} \int_{\tau_\partial(y)\wedge T}^T \bigg(\partial_x\psi(\Phi_t(y),t)\cdot v(\Phi_t(y)) + \partial_t\psi(\Phi_t(y),t)\bigg)\,dt\,d\nu_0(y),
\end{align}
where the \emph{truncation} is defined as $\tau_\partial(y)\wedge T:=\min(\tau_\partial(y), T)$; this is a continuous function in $y$. Interchanging the order of integration is allowed by Fubini's theorem, because the integrand is bounded. 
The subdivision of the domain $[0,T]$ with respect to $\tau_\partial$ is necessary, since the semigroup $\Phi_t$ represents the \emph{stopped} flow and therefore the identity $\frac{d}{dt}\Phi_t(y)=v(\Phi_t(y))$ is only valid if $t\in I_y$. Hence, only in the first integral on the right-hand side of \eqref{eqn:split before after tau}, the chain rule
\begin{equation}
\dfrac{d}{dt}\psi(\Phi_t(y),t) = \partial_x\psi(\Phi_t(y),t)\cdot v(\Phi_t(y)) + \partial_t\psi(\Phi_t(y),t)
\end{equation}
can be used. Note that at time $\tau_\partial(y)\wedge T$ this identity at least holds one-sidedly as $t\nearrow (\tau_\partial(y)\wedge T)$, which is sufficient for the first integral on the right-hand side to be correct.\\
\\
Define, for $z\in\{0,1\}$, the sets
\begin{equation}\label{eqn:def Omega}
\Omega_z^T := \{ y\in[0,1] \,:\, \tau_\partial(y)<T \text{ and } \Phi_{\tau_\partial(y)}(y)=z \}.
\end{equation} 
These are connected subsets of $[0,1]$.\\
If $y\in \Omega^T:=[0,1]\setminus\left(\Omega_0^T\cup\Omega_1^T\right)$, then $\tau_\partial(y) \geqs T$, and obviously
\begin{equation}\label{eqn:integrand zero if tau > T}
\int_{\tau_\partial(y)\wedge T}^T \bigg(\partial_x\psi(\Phi_t(y),t)\cdot v(\Phi_t(y)) + \partial_t\psi(\Phi_t(y),t)\bigg)\,dt\ = 0
\end{equation}
in the second integral on the right-hand side of \eqref{eqn:split before after tau}, since the domain of integration is a nullset (in fact, a single point).\\
If $y\in \Omega_0^T\cup\Omega_1^T$, then $\tau_\partial(y) < T$ and $\Phi_t(y)\in\{0,1\}$ for all $(\tau_\partial(y)\wedge T) \leqs t \leqs T$. Therefore, taking into account the test functions' boundary conditions $\partial_x\psi(0,t)=\partial_x\psi(1,t)=0$, we have that 
\begin{equation}\label{eqn:bdry cond for parts on bdry}
\partial_x\psi(\Phi_t(y),t) = 0
\end{equation}
for all $y\in \Omega_0^T\cup\Omega_1^T$ and $(\tau_\partial(y)\wedge T) \leqs t \leqs T$.\\
\\
Due to \eqref{eqn:integrand zero if tau > T} and \eqref{eqn:bdry cond for parts on bdry}, the second term on the right-hand side of \eqref{eqn:split before after tau} can be written as
\begin{align}\label{eqn:int after tau}
\nonumber & \int_{[0,1]} \int_{\tau_\partial(y)\wedge T}^T \bigg(\partial_x\psi(\Phi_t(y),t)\cdot v(\Phi_t(y)) + \partial_t\psi(\Phi_t(y),t)\bigg)\,dt\,d\nu_0(y)\\
\nonumber & \hspace{0.2\linewidth} = \int_{\Omega_0^T} \int_{\tau_\partial(y)}^T \partial_t\psi(0,t)\,dt\,d\nu_0(y)+ \int_{\Omega_1^T} \int_{\tau_\partial(y)}^T \partial_t\psi(1,t)\,dt\,d\nu_0(y)\\
\nonumber & \hspace{0.2\linewidth} = \int_{\Omega_0^T} \bigg(\psi(\underbrace{0}_{=\Phi_T(y)},T)-\psi\big(0,\tau_\partial(y)\big)\bigg)\,d\nu_0(y)\\
& \hspace{0.2\linewidth} \,\,\,\,+ \int_{\Omega_1^T} \bigg(\psi(\underbrace{1}_{=\Phi_T(y)},T)-\psi\big(1,\tau_\partial(y)\big)\bigg)\,d\nu_0(y).
\end{align}
The first term on the right-hand side of \eqref{eqn:split before after tau} we treat as follows:
\begin{align}\label{eqn:int before tau}
\nonumber &\int_{[0,1]} \int_0^{\tau_\partial(y)\wedge T} \dfrac{d}{dt}\psi(\Phi_t(y),t)\,dt\,d\nu_0(y) \\
\nonumber & \hspace{0.15\linewidth} = \int_{\Omega^T} \int_0^{T} \dfrac{d}{dt}\psi(\Phi_t(y),t)\,dt\,d\nu_0(y)\\
\nonumber &\hspace{0.15\linewidth}\,\,\,\,+\int_{\Omega_0^T} \int_0^{\tau_\partial(y)} \dfrac{d}{dt}\psi(\Phi_t(y),t)\,dt\,d\nu_0(y)+ \int_{\Omega_1^T} \int_0^{\tau_\partial(y)} \dfrac{d}{dt}\psi(\Phi_t(y),t)\,dt\,d\nu_0(y)\\
\nonumber & \hspace{0.15\linewidth} = \int_{\Omega^T} \bigg(\psi\big(\Phi_T(y),T\big)-\psi(y,0)\bigg)\,d\nu_0(y)\\
&\hspace{0.15\linewidth}\,\,\,\,+\int_{\Omega_0^T} \bigg(\psi(0,\tau_\partial(y))-\psi(y,0)\bigg)\,d\nu_0(y)+\int_{\Omega_1^T} \bigg(\psi(1,\tau_\partial(y))-\psi(y,0)\bigg)\,d\nu_0(y).
\end{align}
Note that, for \textit{all} $y\in[0,1]$, the function $t\mapsto \psi(\Phi_t(y),t)$ is differentiable for all $t\geqs0$. If $t=\tau_\partial(y)$, then the differentiability follows from the boundary conditions on $\psi$.\\
\\
Combining \eqref{eqn:split before after tau} with \eqref{eqn:int after tau} and \eqref{eqn:int before tau}, we obtain that
\begin{align}\label{eqn:pair map result}
\nonumber \int_0^T \pair{P_t\,\nu_0}{\partial_x\psi(\cdot,t)\cdot v + \partial_t\psi(\cdot,t)}\,dt =& \int_{[0,1]} \bigg(\psi\big(\Phi_T(y),T\big)-\psi(y,0)\bigg)\,d\nu_0(y)\\
=& \pair{P_T\,\nu_0}{\psi(\cdot,T)} - \pair{\nu_0}{\psi(\cdot,0)}.
\end{align}
Next, consider
\begin{align}\label{eqn:pair perturb}
\nonumber &\int_0^T \pair{\int_0^t P_{t-s}F_f(\mu_s)\,ds}{\partial_x\psi(\cdot,t)\cdot v + \partial_t\psi(\cdot,t)}\,dt \\
\nonumber & \hspace{0.1\linewidth} =  \int_0^T \int_0^t \pair{F_f(\mu_s)}{\partial_x\psi(\Phi_{t-s}(\cdot),t)\cdot v(\Phi_{t-s}(\cdot)) + \partial_t\psi(\Phi_{t-s}(\cdot),t)}\,ds\,dt\\
& \hspace{0.1\linewidth} =  \int_0^T \pair{F_f(\mu_s)}{\int_s^T \bigg(\partial_x\psi(\Phi_{t-s}(\cdot),t)\cdot v(\Phi_{t-s}(\cdot)) + \partial_t\psi(\Phi_{t-s}(\cdot),t)\bigg)\,dt}\,ds
\end{align}
We subdivide the domain of the spatial integration into $\Omega_0^{T-s}$, $\Omega_1^{T-s}$ and $\Omega^{T-s}:=[0,1]\setminus(\Omega_0^{T-s}\cup\Omega_1^{T-s})$, with the sets $\Omega_z^{T-s}$ defined analogous to \eqref{eqn:def Omega}.\\
For each $z\in\{0,1\}$, we have
\begin{align}\label{eqn:int perturb hit bdry}
\nonumber & \int_0^T \int_{\Omega_z^{T-s}} \int_s^T \bigg(\partial_x\psi(\Phi_{t-s}(y),t)\cdot v(\Phi_{t-s}(y)) + \partial_t\psi(\Phi_{t-s}(y),t)\bigg)\,dt \, dF_f(\mu_s)(y)\,ds \\
\nonumber & \hspace{0.2\linewidth} =  \int_0^T \int_{\Omega_z^{T-s}} \int_s^{\tau_\partial(y)+s} \dfrac{d}{dt}\psi(\Phi_{t-s}(y),t)\,dt \, dF_f(\mu_s)(y)\,ds \\
\nonumber & \hspace{0.2\linewidth} \,\,\,\, + \int_0^T \int_{\Omega_z^{T-s}} \int_{\tau_\partial(y)+s}^{T} \bigg(\underbrace{\partial_x\psi(z,t)}_{=0}\cdot v(z) + \underbrace{\partial_t\psi(z,t)}_{=\frac{d}{dt}\psi(z,t)}\bigg)\,dt \, dF_f(\mu_s)(y)\,ds\\
\nonumber & \hspace{0.2\linewidth} =  \int_0^T \int_{\Omega_z^{T-s}} \bigg(\psi\big(z,\tau_\partial(y)+s\big)- \psi(y,s) \bigg) \, dF_f(\mu_s)(y)\,ds \\
\nonumber & \hspace{0.2\linewidth} \,\,\,\, + \int_0^T \int_{\Omega_z^{T-s}} \bigg(\psi(z,T)-\psi\big(z,\tau_\partial(y)+s\big)\bigg) \, dF_f(\mu_s)(y)\,ds\\
& \hspace{0.2\linewidth} =  \int_0^T \int_{\Omega_z^{T-s}} \bigg(\psi\big(\Phi_{T-s}(y),T\big)- \psi(y,s) \bigg) \, dF_f(\mu_s)(y)\,ds.
\end{align}
Considering the spatial domain of integration $\Omega^{T-s}$, we find
\begin{align}\label{eqn:int perturb interior}
\nonumber & \int_0^T \int_{\Omega^{T-s}} \int_s^T \bigg(\partial_x\psi(\Phi_{t-s}(y),t)\cdot v(\Phi_{t-s}(y)) + \partial_t\psi(\Phi_{t-s}(y),t)\bigg)\,dt \, dF_f(\mu_s)(y)\,ds \\
\nonumber & \hspace{0.2\linewidth} =  \int_0^T \int_{\Omega^{T-s}} \int_s^{T} \dfrac{d}{dt}\psi(\Phi_{t-s}(y),t)\,dt \, dF_f(\mu_s)(y)\,ds \\
& \hspace{0.2\linewidth} =  \int_0^T \int_{\Omega^{T-s}} \bigg(\psi\big(\Phi_{T-s}(y),T\big)- \psi(y,s) \bigg) \, dF_f(\mu_s)(y)\,ds.
\end{align}
Together, \eqref{eqn:pair perturb}, \eqref{eqn:int perturb hit bdry} and \eqref{eqn:int perturb interior} yield
\begin{align}\label{eqn:pair perturb result}
\nonumber \int_0^T \pair{\int_0^t P_{t-s}F_f(\mu_s)\,ds}{\partial_x\psi(\cdot,t)\cdot v + \partial_t\psi(\cdot,t)}\,dt =&
\int_0^T \pair{F_f(\mu_s)}{\psi\big(\Phi_{T-s}(\cdot),T\big)}\,ds\\
& - \int_0^T\pair{F_f(\mu_s)}{\psi(\cdot,s)}\,ds.
\end{align}
It follows from \eqref{eqn:pair map result} and \eqref{eqn:pair perturb result}, and from the variation of constants formula \eqref{eq:VOC} that
\begin{align}\label{eqn:combine nearly def weak}
\nonumber &\int_0^T \pair{\mu_t}{\partial_x\psi(\cdot,t)\cdot v + \partial_t\psi(\cdot,t)}\,dt = \pair{P_T\,\nu_0}{\psi(\cdot,T)} - \pair{\nu_0}{\psi(\cdot,0)}\\
& \hspace{0.25\linewidth}+ \int_0^T \pair{F_f(\mu_s)}{\psi\big(\Phi_{T-s}(\cdot),T\big)}\,ds - \int_0^T\pair{F_f(\mu_s)}{\psi(\cdot,s)}\,ds.
\end{align}
Note that
\begin{align*}
 \int_0^T \pair{F_f(\mu_s)}{\psi\big(\Phi_{T-s}(\cdot),T\big)}\,ds =&   \int_0^T \pair{P_{T-s}\,F_f(\mu_s)}{\psi(\cdot,T)}\,ds\\
 =& \pair{\int_0^T P_{T-s}\,F_f(\mu_s)\,ds}{\psi(\cdot,T)},
\end{align*}
and hence
\begin{equation*}
\pair{P_T\,\nu_0}{\psi(\cdot,T)} + \int_0^T \pair{F_f(\mu_s)}{\psi\big(\Phi_{T-s}(\cdot),T\big)}\,ds = \pair{\mu_T}{\psi(\cdot,T)}.
\end{equation*}
Equation \eqref{eqn:combine nearly def weak} can thus be written as
\begin{equation*}
\nonumber \int_0^T \pair{\mu_t}{\partial_x\psi(\cdot,t)\cdot v + \partial_t\psi(\cdot,t)}\,dt = \pair{\mu_T}{\psi(\cdot,T)} - \pair{\nu_0}{\psi(\cdot,0)} - \int_0^T\pair{F_f(\mu_s)}{\psi(\cdot,s)}\,ds,
\end{equation*}
which shows that $\mu$ is a weak solution of \eqref{eq:main equation}.
\end{proof}

\section{Mild and weak solutions for measure-dependent velocity}\label{sec: gen meas-dep v}
In this section we summarize the results of \cite{EvHiMu_SIMA} and compare the concept of mild solutions from \cite{EvHiMu_SIMA} to weak solutions of \eqref{eqn:main equation nonlin}. In \cite{EvHiMu_SIMA}, we generalized the assumptions on $v$ from \cite{EvHiMu_JDE} in the following way to measure-dependent velocity fields:
\begin{assumption}[Assumptions on the measure-dependent velocity field]\label{ass: v properties}
Assume that $\map{v}{\CM([0,1])\times[0,1]}{\R}$ is a mapping such that:
\begin{enumerate}
  \item[(i)] $v[\mu]\in\BL([0,1])$, for each $\mu\in\CM([0,1])$.\label{ass: part v meas-dep in BL}
\end{enumerate}
Furthermore, assume that for any $R>0$ there are constants $K_R$, $L_R$, $M_R$ such that for all $\mu,\nu\in\CM([0,1])$ satisfying $\|\mu\|_\TV\leqslant R$ and $\|\nu\|_\TV\leqslant R$, the following estimates hold:
\begin{enumerate}
  \item[(ii)] $\|v[\mu]\|_\infty\leqs K_R$,\label{ass: part unif bnd}
  \item[(iii)] $|\,v[\mu]\,|_\Lip\leqs L_R$, \,\,\, and\label{ass: part Lipsch in x}
  \item[(iv)] $\|v[\mu]-v[\nu]\|_\infty \leqslant M_R\,\|\mu-\nu\|^*_\BL$.\label{ass: part Lipsch in mu incl R}
\end{enumerate}
\end{assumption}
\noindent In \cite{EvHiMu_SIMA}, we proved well-posedness of \eqref{eqn:main equation nonlin}:
\begin{equation*}
\frac{\partial}{\partial t}\mu_t + \frac{\partial}{\partial x} (v[\mu_t]\,\mu_t) = F_f(\mu_t)
\end{equation*}
on $[0,1]$, in the sense of mild solutions; cf.~Definition \ref{def: mild soln nonlin} and Theorem \ref{thm: exist uniq nonlin}. Like in \cite{EvHiMu_SIMA}, in this paper we restrict ourselves to $f$ that is bounded Lipschitz on $[0,1]$. See Section \ref{sec: discussion} for further discussion on this assumption.\\
\\
Mild solutions for measure-dependent $v=v[\mu]$ are defined using mild solutions for fixed $v\in\BL([0,1])$ as a building block via an Euler-like approach. We first summarize the required notation used in \cite{EvHiMu_SIMA}.\\
\\ 
Let $v\in\BL([0,1])$ and $f\in\BL([0,1])$ be arbitrary. For all $t\geqs0$, we define $\map{Q_t}{\CM([0,1])}{\CM([0,1])}$ to be the operator that maps the initial condition to the solution in the sense of Definition \ref{def: mild soln fixed v}. Theorem \ref{thm:exist uniq v prescribed} guarantees that this operator is well-defined and continuous for $\|\cdot\|^*_\BL$. Moreover, $Q$ preserves positivity, due to \cite[Corollary 3.4]{EvHiMu_JDE}. The set of operators $(Q_t)_{t\geqs0}$ constitutes a semigroup and has useful other properties, like certain Lipschitz estimates. The exact results and their proofs can be found in \cite[Section 2.3]{EvHiMu_SIMA}.\\
In the sequel we will write e.g.~$Q^v$ and $Q^{v'}$ to distinguish between the semigroups $Q$ on $\CM([0,1])$ associated to $v\in\BL([0,1])$  and $v'\in\BL([0,1])$, respectively.\\
\\
We now introduce the aforementioned forward-Euler-like approach to construct approximate solutions. Let $T>0$ be given. Let $N\geqslant 1$ be fixed and define a set $\alpha\subset [0,T]$ as follows: 
\begin{equation}\label{eqn: partition}
\alpha:= \big\{ t_j\in[0,T] \;\, : \;\, 0\leqslant j \leqslant N,\; t_0=0,\; t_{N}=T,\; t_j<t_{j+1}  \big\},
\end{equation}
which we call a \textit{partition} of the interval $[0,T]$. Here, $N$ denotes the number of \textit{subintervals} in $\alpha$.\\
\\
Let $\mu_0\in\CM([0,1])$ be fixed. For a given partition $\alpha:=\{t_0,\ldots,t_N\}\subset[0,T]$, define a measure-valued trajectory $\mu \in C([0,T];\CM([0,1]))$ by 
\begin{equation}\label{eqn: Euler scheme}
\left\{
  \begin{array}{ll}
    \mu_t := Q^{v_j}_{t-t_j}\,\mu_{t_j},&\text{if }t\in(t_j,t_{j+1}];\vspace{0.2cm} \\ 
    v_j:=v[\mu_{t_j}];&\vspace{0.2cm}\\
    \mu_{t=0}=\mu_0,&
  \end{array}
\right.
\end{equation}
for all $j\in\{0,\ldots,N-1\}$. Here, $(Q^v_t)_{t\geqslant0}$ denotes the semigroup introduced above. Note that by Assumption \ref{ass: v properties}, Part (i), $v_j=v[\mu_{t_j}]\in\BL([0,1])$ for each $j$.\\     
\\
We call this a forward-Euler-like approach, because it is the analogon of the forward Euler method for ODEs (cf.~e.g.~\cite[Chapter 2]{Butcher}). See \cite[Section 3]{EvHiMu_SIMA} for further explanation.\\
\\
The conditions in Parts (ii)--(iv) 
of Assumption \ref{ass: v properties} are only required to hold for measures in a $\TV$-norm bounded set, in view of the following lemma:

\begin{lemma}\label{lem: set timeslices bdd}
Let $\mu_0\in\CM([0,1])$ be given and let $\map{v}{\CM([0,1])\times[0,1]}{\R}$ satisfy Assumption \ref{ass: v properties}(i). For a given partition $\alpha:=\{t_0,\ldots,t_N\}\subset[0,T]$, let $\mu \in C([0,T];\CM([0,1]))$ be defined by \eqref{eqn: Euler scheme}. Then for all $t\in[0,T]$
\begin{equation*}
\|\mu_t\|^*_\BL  \;\leqs\; \|\mu_t\|_\TV \;\leqslant\; \|\mu_0\|_\TV \,\exp(\|f\|_\infty\,T).
\end{equation*}
This bound is in particular independent of $t$, $N$ and the distribution of points within $\alpha$.
\end{lemma}
\begin{proof}
See the proof of \cite[Lemma 3.4]{EvHiMu_SIMA} for details.
\end{proof}
\noindent We construct sequences of Euler approximations, each following from a sequence of partitions $(\alpha_k)_{k\in\N}$ that satisfies the following assumption:
\begin{assumption}[Assumptions on the sequence of partitions]\label{ass: partition}
Let $(\alpha_k)_{k\in\N}$ be a sequence of partitions of $[0,T]$ and let $(N_k)_{k\in\N}\subset\N$ be the corresponding sequence such that each $\alpha_k$ is of the form
\begin{equation}\label{eqn: partition alpha_k}
\alpha_k:= \big\{ t^k_j\in[0,T] \;\, : \;\, 0\leqslant j \leqslant N_k,\; t^k_0=0,\; t^k_{N_k}=T,\; t^k_j<t^k_{j+1}  \big\}.
\end{equation}
Define 
\begin{equation}
M^{(k)}:= \max_{j\in\{0,\ldots,N_k-1\}}t^k_{j+1}-t^k_j \label{eqn: def max interval}
\end{equation}
for all $k\in\N$. Assume that the sequence $(M^{(k)})_{k\in\N}$ is nonincreasing and $M^{(k)}\to 0$ as $k\to \infty$.
\end{assumption}
\noindent A mild solution is defined as follows:
\begin{definition}[See {\cite[Definition 3.8]{EvHiMu_SIMA}}]\label{def: mild soln nonlin}
Let the space of continuous maps from $[0,T]$ to $\CM([0,1])$ be endowed with the metric defined for all $\mu,\nu\in C([0,T]; \CM([0,1]))$ by
\begin{equation}\label{eqn: metric sup BL norm}
\sup_{t\in[0,T]}\|\mu_t-\nu_t\|^*_\BL.
\end{equation}
Let $(\alpha_k)_{k\in\N}$ be a sequence of partitions satisfying Assumption \ref{ass: partition}. For each $k\in\N$, let $\mu^k\in C([0,T];\CM([0,1]))$ be defined by \eqref{eqn: Euler scheme} with partition $\alpha_k$. Then, for any such sequence of partitions $(\alpha_k)_{k\in\N}$, any limit of a subsequence of $(\mu^k)_{k\in\N}$ is called a \textit{(measure-valued) mild solution} of \eqref{eqn:main equation nonlin}.
\end{definition}
\noindent The name \textit{mild solutions} is appropriate, first of all because they are constructed from piecewise mild solutions in the sense of Definition \ref{def: mild soln fixed v}.\\
Moreover, as their name suggests, \textit{weak} solutions in general constitute a weaker solution concept than mild solutions. That is, a mild solution (meaning: a solution of the variation of constants formula) is in general necessarily a weak solution.  See e.g.~\cite[p.~4--5]{Evers_PhD}, \cite[Proposition 3.7]{Hoogwater-thesis}, Theorem \ref{thm:mild is weak v=v(x)} in this paper, and the way in which mild and weak solutions are connected on \cite[pp.~258--259]{Pazy}. In all of these references, the equations treated are simpler than \eqref{eqn:main equation nonlin} that is considered in this section. Here, mild solutions are the ones constructed in \cite{EvHiMu_SIMA}, being defined as the limit of Euler approximations. They are not solutions of the variation of constants formula. However, they are still elements of the set of weak solutions, as we will show in Theorem \ref{thm:mild is weak v=v(mu)}. This implication is an extra justification for the name \textit{mild} solutions.\\
\\
In the rest of this paper we focus on \textit{positive} measure-valued solutions, because these are the only physically relevant solutions in many applications. 
\begin{theorem}[Existence and uniqueness of mild solutions to \eqref{eqn:main equation nonlin}]\label{thm: exist uniq nonlin}
Let $\mu_0\in\CM^+([0,1])$ be given and let $\map{v}{\CM([0,1])\times[0,1]}{\R}$ satisfy Assumption \ref{ass: v properties}. Endow the space $C([0,T];\CM([0,1]))$ with the metric defined by \eqref{eqn: metric sup BL norm}. Then, there is a unique element of $C([0,T]; \CM^+([0,1]))$ with initial condition $\mu_0$, that is a mild solution in the sense of Definition \ref{def: mild soln nonlin}. That is, for each sequence of partitions $(\alpha_k)_{k\in\N}$ satisfying Assumption \ref{ass: partition}, the corresponding sequence $(\mu^k)_{k\in\N}$ defined by \eqref{eqn: Euler scheme} is a sequence in $C([0,T];\CM^+([0,1]))$ and has a unique limit as $k\to\infty$.\\
Moreover, this limit is independent of the choice of $(\alpha_k)_{k\in\N}$.
\end{theorem}
\begin{proof}
See \cite[Theorem 3.10]{EvHiMu_SIMA} for details.
\end{proof}

\begin{definition}[Weak solution to \eqref{eqn:main equation nonlin}]\label{def: weak soln nonlin}
Fix $T\geqs0$, let $f\in\BL([0,1])$ and let $\map{v}{\CM([0,1])\times[0,1]}{\R}$ satisfy Assumption \ref{ass: v properties}. Then $\map{\mu}{[0,T]}{\CM([0,1])}$ is a \emph{weak solution} to \eqref{eqn:main equation nonlin} corresponding to initial condition $\nu_0\in\CM([0,1])$, if 
\begin{equation}
\pair{\mu_T}{\psi(\cdot,T)}-\pair{\nu_0}{\psi(\cdot,0)} = \int_0^T \pair{\mu_t}{\partial_t\psi(\cdot,t) + \partial_x\psi(\cdot,t)\cdot v[\mu_t]} \, dt + \int_0^T \pair{F_f(\mu_t)}{\psi(\cdot,t)}\,dt
\end{equation}
is satisfied for all $\psi\in\Lambda^T$, with $\Lambda^T$ as defined in \eqref{eqn:weak v=v(x) test funct}.
\end{definition}
\noindent To show that mild solutions in the sense of Definition \ref{def: mild soln nonlin} are weak solutions in the sense of Definition \ref{def: weak soln nonlin}, the following result is useful.
\begin{lemma}[Cf.~portmanteau theorem]\label{lem:Portemanteau}
Let $(\nu^k)_{k\in\N}\subset \CM([0,1])$ and assume there is an $R>0$ such that $\nu^k([0,1])\leqs R$ for all $k\in\N$. Let $\nu\in\CM([0,1])$. The following are equivalent:
\begin{enumerate}[(a)]
\item $\pair{\nu^k}{\phi}\to\pair{\nu}{\phi}$ as $k\to\infty$ for all $\phi\in\BL([0,1])$;\label{part:Portemanteau BL}
\item $\pair{\nu^k}{\phi}\to\pair{\nu}{\phi}$ as $k\to\infty$ for all $\phi\in C_b([0,1])$.\label{part:Portemanteau Cb}
\end{enumerate}
\end{lemma}
\begin{proof}
Apply \cite[Theorem 13.16]{Klenke} to the sequence $(\nu^k/R)_{k\in\N}\subset \CM_{\leqs 1}([0,1])$, and use the equivalence ``(ii)$\Leftrightarrow$(iii)" therein, which implies the equivalence with (\ref{part:Portemanteau Cb}) above.
\end{proof}
\noindent The main result of this paper is the following theorem.
\begin{theorem}[Mild solutions to \eqref{eqn:main equation nonlin} are weak solutions]\label{thm:mild is weak v=v(mu)}
Let $\map{\mu}{[0,T]}{\CM^+([0,1])}$ be the mild solution provided by Theorem \ref{thm: exist uniq nonlin}, corresponding to initial value $\nu_0\in\CM^+([0,1])$. Then, $\mu$ is a weak solution of \eqref{eqn:main equation nonlin}.
\end{theorem}
\noindent The proof of Theorem \ref{thm:mild is weak v=v(mu)} uses ideas from the proof of \cite[Proposition 4.9]{Hoogwater-thesis}.
\begin{proof}
Let $(\alpha_k)_{k\in\N}$ be a sequence of partitions satisfying Assumption \ref{ass: partition}, and for each $k\in\N$, let $\mu^k\in C([0,T];\CM([0,1]))$ be the corresponding sequence of Euler approximations defined by \eqref{eqn: Euler scheme}. Since $\mu$ is the unique mild solution, $\mu=\lim_{k\to\infty}\mu^k$ holds with convergence in the metric given by \eqref{eqn: metric sup BL norm}.\\
\\
Fix a $\psi\in\Lambda^T$. For each $j\in\{0,\ldots,N_k-1\}$ we apply Theorem \ref{thm:mild is weak v=v(x)} to subinterval $[t^k_j,t^k_{j+1}]$ in the Euler approximation. Note that the restriction $\psi|_{[t^k_j,t^k_{j+1}]}$ is an appropriate test function on the corresponding domain $[0,1]\times[t^k_j,t^k_{j+1}]$. The individual test functions on these subdomains are therefore all derived from the \textit{same} $\psi$ on $[0,1]\times[0,T]$. Moreover, the required regularity per subdomain and the spatial boundary conditions are simply inherited from $\psi$.\\
For each $j\in\{0,\ldots,N_k-1\}$ we thus have
\begin{align}
\nonumber \pair{\mu^k_{t^k_{j+1}}}{\psi(\cdot,t^k_{j+1})}-\pair{\mu^k_{t^k_j}}{\psi(\cdot,t^k_j)} =& \int_{t^k_j}^{t^k_{j+1}} \pair{\mu^k_t}{\partial_t\psi(\cdot,t) + \partial_x\psi(\cdot,t)\cdot v[\mu^k_{t^k_j}]} \, dt \\
&+ \int_{t^k_j}^{t^k_{j+1}} \pair{F_f(\mu^k_t)}{\psi(\cdot,t)}\,dt,
\end{align}
while $\mu^k_{t^k_0}=\mu^k_0=\nu_0$. Note that within the integrals we do not need to write $\psi|_{[t^k_j,t^k_{j+1}]}$, but we can simply use $\psi$. Summation over $j\in\{0,\ldots,N_k-1\}$ yields
\begin{align}\label{eqn:sum weak on subintervals}
\nonumber \pair{\mu^k_T}{\psi(\cdot,T)}-\pair{\nu_0}{\psi(\cdot,0)} =& \sum_{j=0}^{N_k-1}\int_{t^k_j}^{t^k_{j+1}} \pair{\mu^k_t}{\partial_t\psi(\cdot,t) + \partial_x\psi(\cdot,t)\cdot v[\mu^k_{t^k_j}]} \, dt\\
\nonumber & + \int_0^T \pair{F_f(\mu^k_t)}{\psi(\cdot,t)}\,dt\\
\nonumber =& \int_0^T \pair{\mu^k_t}{\partial_t\psi(\cdot,t) + \partial_x\psi(\cdot,t)\cdot \bar{v}^k(t,\cdot)} \, dt \\
&+ \int_0^T \pair{F_f(\mu^k_t)}{\psi(\cdot,t)}\,dt.
\end{align}
Here, $\map{\bar{v}^k}{[0,T]}{\BL([0,1])}$ is defined by $\bar{v}^k(t,\cdot):=v[\mu^k_{t^k_j}]$ whenever $t\in(t^k_j,t^k_{j+1}]$, while $\bar{v}^k(0,\cdot)=v[\nu_0]$.\\
\\
Note that, for each $t\in[0,T]$ fixed, $\psi(\cdot,t)\in \BL([0,1])$ due to the assumed regularity on the test functions. Since $\mu^k\to\mu$ with respect to the metric in \eqref{eqn: metric sup BL norm} as $k\to\infty$, we have in particular that
\begin{equation}
\pair{\mu^k_T}{\psi(\cdot,T)} \stackrel{k\to\infty}{\longrightarrow} \pair{\mu_T}{\psi(\cdot,T)}.\label{eqn:conv pair at T}
\end{equation}
The second term on the right-hand side of \eqref{eqn:sum weak on subintervals} we treat as follows:
\begin{align}
\nonumber \left|\int_0^T \pair{F_f(\mu^k_t)}{\psi(\cdot,t)}\,dt - \int_0^T \pair{F_f(\mu_t)}{\psi(\cdot,t)}\,dt\right| \leqs& \int_0^T \left|\pair{\mu^k_t-\mu_t}{\psi(\cdot,t)\cdot f(\cdot)}\right|\,dt\\
\leqs& \int_0^T \|\psi(\cdot,t)\|_\BL\cdot\|f\|_\BL\cdot\|\mu^k_t-\mu_t\|^*_\BL\,dt,\label{eqn:est conv Ff pair int with f BL}
\end{align}
where we used that for each $t\in[0,T]$ the product $\psi(\cdot,t)\cdot f(\cdot)$ is bounded Lipschitz and $\|\psi(\cdot,t)\cdot f(\cdot)\|_\BL \leqs \|\psi(\cdot,t)\|_\BL\cdot\|f\|_\BL$; cf.~\eqref{eqn:Banach algebra}.\\
\\
Since $\psi\in C^1([0,1]\times[0,T])$, we know that
\begin{equation*}
\sup_{[0,1]\times[0,T]}|\psi|<\,\infty, \,\,\,\, \text{ and } \,\,\,\,
\sup_{[0,1]\times[0,T]}|\partial_x\psi|<\,\infty,
\end{equation*}
and hence,
\begin{equation*}
\sup_{t\in[0,T]}\|\psi(\cdot,t)\|_\BL <\infty.
\end{equation*}
Consequently, we have
\begin{align}
\nonumber &\left|\int_0^T \pair{F_f(\mu^k_t)}{\psi(\cdot,t)}\,dt - \int_0^T \pair{F_f(\mu_t)}{\psi(\cdot,t)}\,dt \right| \\
&\hspace{0.3\linewidth}\leqs  \sup_{t\in[0,T]}\|\psi(\cdot,t)\|_\BL\cdot\|f\|_\BL\cdot\underbrace{\sup_{t\in[0,T]}\|\mu^k_t-\mu_t\|^*_\BL}_{\to0 \text{ as }k\to\infty}\cdot \, T,
\end{align}
and thus 
\begin{equation}
\int_0^T \pair{F_f(\mu^k_t)}{\psi(\cdot,t)}\,dt \stackrel{k\to\infty}{\longrightarrow} \int_0^T \pair{F_f(\mu_t)}{\psi(\cdot,t)}\,dt .\label{eqn:conv pair Ff}
\end{equation}
Finally, we consider the first term on the right-hand side of \eqref{eqn:sum weak on subintervals}.\\
For each $t\in[0,T]$, it holds that $\|\mu^k_t-\mu_t\|^*_\BL\to 0$ as $k\to\infty$ since the convergence $\mu^k\to\mu$ is in the metric \eqref{eqn: metric sup BL norm}. Then in particular, $\pair{\mu^k_t}{\phi}\to\pair{\mu_t}{\phi}$ for all $\phi\in\BL([0,1])$. Since $\mu^k_t([0,1])=\|\mu^k_t\|_\TV$, Lemma \ref{lem: set timeslices bdd} provides a bound on $\mu^k_t([0,1])$ that is independent of $k$ and $t$. We can therefore apply Lemma \ref{lem:Portemanteau}, and conclude that 
\begin{equation}\label{eqn: mukt conv weakly against Cb}
\pair{\mu^k_t}{\phi}\to\pair{\mu_t}{\phi},\,\,\,\,\text{ as } k\to\infty,
\end{equation}
for all $\phi\in C_b([0,1])$.\\
For each $t\in[0,T]$, we have that $\partial_t\psi(\cdot,t),\partial_x\psi(\cdot,t)\in C_b([0,1])$ by the assumption that $\psi\in C^1_b([0,1]\times[0,T])$. Moreover, $v[\mu_t]\in\BL([0,1])$ and thus $\partial_x\psi(\cdot,t)\cdot v[\mu_t]\in C_b([0,1])$. Hence, \eqref{eqn: mukt conv weakly against Cb} implies that
\begin{align}
\pair{\mu^k_t}{\partial_t\psi(\cdot,t)}&\to \pair{\mu_t}{\partial_t\psi(\cdot,t)}   \text{, and} \label{eqn:conv pair muk mu against dtPsi}\\
\pair{\mu^k_t}{\partial_x\psi(\cdot,t)\cdot v[\mu_t]}&\to \pair{\mu_t}{\partial_x\psi(\cdot,t)\cdot v[\mu_t]} \label{eqn:conv pair muk mu against dxPsi vmu}
\end{align}
as $k\to\infty$.\\
\\
For each $t\in[0,T]$, the function $\partial_x\psi(\cdot,t)\cdot \left( \bar{v}^k(t,\cdot) -  v[\mu_t]\right)$ is in $C_b([0,1])$ and therefore
\begin{align}
\nonumber &\left| \pair{\mu^k_t}{\partial_x\psi(\cdot,t)\cdot \bar{v}^k(t,\cdot)} - \pair{\mu^k_t}{\partial_x\psi(\cdot,t)\cdot v[\mu_t]} \right| \\
\nonumber &\hspace{0.3\linewidth}= \left|\, \int_{[0,1]}\partial_x\psi(x,t)\cdot \left( \bar{v}^k(t,x) -  v[\mu_t](x)\right)\,d\mu^k_t(x) \, \right| \\
&\hspace{0.3\linewidth}\leqs \sup_{\tau\in[0,T]}\|\partial_x\psi(\cdot,\tau)\|_\infty\cdot \| \bar{v}^k(t,\cdot) -  v[\mu_t] \|_\infty \cdot \|\mu^k_t\|_\TV.\label{eqn:est pair vmuk against vmu}
\end{align}
Due to Lemma \ref{lem: set timeslices bdd}, it holds for each $t\in[0,T]$ that $\|\mu^k_t\|_\TV\leqs R:=\|\nu_0\|_\TV\,\exp(\|f\|_\infty\,T)$. The measure $\mu_t$ is positive for all $t\in[0,T]$, and thus
\begin{equation}
\|\mu_t\|_\TV = \|\mu_t\|^*_\BL \leqs \underbrace{\|\mu^k_t\|^*_\BL}_{\leqs R} + \underbrace{\|\mu^k_t-\mu_t\|^*_\BL}_{\to 0},
\end{equation}
whence $\|\mu_t\|_\TV\leqs R$. We can now use Assumption \ref{ass: v properties}(iv) to estimate the term $\| \bar{v}^k(t,\cdot) -  v[\mu_t] \|_\infty$.\\
Without loss of generality, assume that $t>0$. Let $j\in\{0,\ldots,N_k\}$ be such that $t^k_j<t\leqs t^k_{j+1}$. Then
\begin{align}
\nonumber \| \bar{v}^k(t,\cdot) -  v[\mu_t] \|_\infty = \| v[\mu^k_{t^k_j}] -  v[\mu_t] \|_\infty \leqs&\, \|v[\mu^k_{t^k_j}] -  v[\mu_{t^k_j}] \|_\infty + \|v[\mu_{t^k_j}] -  v[\mu_t] \|_\infty\\
\nonumber \leqs&\, M_R\, \|\mu^k_{t^k_j} -  \mu_{t^k_j} \|^*_\BL +M_R\, \|\mu_{t^k_j} -  \mu_t \|^*_\BL\\
\leqs&\, M_R\, \sup_{\tau\in[0,T]}\|\mu^k_{\tau} -  \mu_{\tau} \|^*_\BL +M_R\, \|\mu_{t^k_j} -  \mu_t \|^*_\BL.\label{eqn: vk - vmu est}
\end{align}
The first term on the right-hand side goes to zero, because of the uniform convergence of $\mu^k$ to $\mu$. The second term on the right-hand side goes to zero because $t\mapsto\mu_t$ is continuous and $t-t^k_j\leqs M^{(k)}\to 0$ as $k\to\infty$. Consequently, the left-hand side of \eqref{eqn: vk - vmu est} must vanish as $k\to\infty$.\\
It follows from \eqref{eqn:est pair vmuk against vmu} and the fact that the left-hand side of \eqref{eqn: vk - vmu est} goes to zero, that 
\begin{equation}
\left|\pair{\mu^k_t}{\partial_x\psi(\cdot,t)\cdot \bar{v}^k(t,\cdot)} - \pair{\mu^k_t}{\partial_x\psi(\cdot,t)\cdot v[\mu_t]} \right| \to 0, \text{ as } k\to \infty.
\end{equation}
This results yields, together with \eqref{eqn:conv pair muk mu against dxPsi vmu}, that for all $t\in[0,T]$
\begin{equation}
\pair{\mu^k_t}{\partial_x\psi(\cdot,t)\cdot \bar{v}^k(t,\cdot)}\to \pair{\mu_t}{\partial_x\psi(\cdot,t)\cdot v[\mu_t]}, \text{ as } k\to\infty, \label{eqn:conv muk vmuk to mu vmu}
\end{equation}
due to the triangle inequality.\\
\\
Combining \eqref{eqn:conv pair muk mu against dtPsi} and \eqref{eqn:conv muk vmuk to mu vmu}, we obtain that for all $t\in[0,T]$
\begin{equation}
\pair{\mu^k_t}{\partial_t\psi(\cdot,t)+\partial_x\psi(\cdot,t)\cdot \bar{v}^k(t,\cdot)}\to \pair{\mu_t}{\partial_t\psi(\cdot,t)+\partial_x\psi(\cdot,t)\cdot v[\mu_t]}, \text{ as } k\to\infty. \label{eqn:conv pair pw t}
\end{equation}
Since for each $t\in[0,T]$ and $k\in\N$, the function $\partial_t\psi(\cdot,t)+\partial_x\phi(\cdot,t)\cdot \bar{v}^k(t,\cdot)$ is an element of $C_b([0,1])$, it holds that
\begin{equation}
\left| \pair{\mu^k_t}{\partial_t\psi(\cdot,t)+\partial_x\phi(\cdot,t)\cdot \bar{v}^k(t,\cdot)} \right| \leqs \left( \sup_{\tau\in[0,T]}\|\partial_t\psi(\cdot,\tau)\|_\infty + \sup_{\tau\in[0,T]}\|\partial_x\psi(\cdot,\tau)\|_\infty\cdot K_R \right)\cdot R,\label{eqn:pair bound}
\end{equation}
with the same $R=\|\nu_0\|_\TV\,\exp(\|f\|_\infty\,T)$ as before. Here we used Lemma \ref{lem: set timeslices bdd} and Assumption \ref{ass: v properties}(ii).\\
\\
By the dominated convergence theorem, cf.~e.g.~\cite[Theorem 4.2]{Brezis}, we now obtain from \eqref{eqn:conv pair pw t} and \eqref{eqn:pair bound} in particular that 
\begin{equation}
\int_0^T\pair{\mu^k_t}{\partial_t\psi(\cdot,t)+\partial_x\psi(\cdot,t)\cdot \bar{v}^k(t,\cdot)}\,dt \to \int_0^T \pair{\mu_t}{\partial_t\psi(\cdot,t)+\partial_x\psi(\cdot,t)\cdot v[\mu_t]}\,dt, \text{ as } k\to\infty. \label{eqn:conv pair int}
\end{equation}
By taking the limit $k\to\infty$ in \eqref{eqn:sum weak on subintervals}, while taking \eqref{eqn:conv pair at T}, \eqref{eqn:conv pair Ff} and \eqref{eqn:conv pair int} into account, we obtain
\begin{align}
\nonumber \pair{\mu_T}{\psi(\cdot,T)}-\pair{\nu_0}{\psi(\cdot,0)} =& \int_0^T \pair{\mu_t}{\partial_t\psi(\cdot,t) + \partial_x\psi(\cdot,t)\cdot v[\mu_t]} \, dt + \int_0^T \pair{F_f(\mu_t)}{\psi(\cdot,t)}\,dt.
\end{align}
Since $\psi\in\Lambda^T$ is chosen arbitrarily, this proves the statement of the theorem.
%
\end{proof}

\section{Context of the results and open issues}\label{sec: discussion}
In the Discussion section of \cite{EvHiMu_SIMA}, we explained that in fact we would like to consider mild solutions (with measure-dependent velocity) corresponding to a sequence $(f_n)_{n\in\N}\subset\BL([0,1])$, such that $f_n\to f$ pointwise, and $f$ is \textit{piecewise} bounded Lipschitz. Let $(\mu^{f_n})_{n\in\N}$ denote the corresponding sequence of measure-valued mild solutions, with measure-dependent velocity $v=v[\mu^{f_n}]$.\\
In \cite{EvHiMu_JDE} we specifically focussed on such sequence $(f_n)_{n\in\N}$ that describes a vanishing boundary layer in which mass is gated away from the domain. Assume there are regions around $0$ and $1$ in which mass decays, and that these regions shrink to zero width. That is, $f_n$ is nonzero only in a region around $x=0$ and $x=1$, respectively. Moreover, this region shrinks to zero as $n\to\infty$ and $f_n\to f$, where $f$ satisfies $f(x)=0$ if $x\in(0,1)$ and $f(0)=f(1)=-1$.\\
\\
We want to know whether the sequence $(\mu^{f_n})_{n\in\N}$ converges, and whether the limit coincides with the mild solution corresponding to $f$ (if this mild solution exists). Mild solutions were obtained in \cite{EvHiMu_SIMA} as the limit of Euler approximations. Let the approximating sequence corresponding to $f_n$ be indexed by $k$, and let $\mu^{f_n,k}$ be one such Euler approximation. The question is now whether the limits $k\to\infty$ and $n\to\infty$ commute. The following scheme shows the four limit processes involved:
\begin{center}
\begin{tabular}{llcrr}
&&\textbf{\small(A)}&&\\
\\
&$\boxed{\mu^{f_n,k}}$ &  $\stackrel{n\to\infty}{\longrightarrow}$ & $\boxed{\mu^{f,k}}$&\\
\\
\textbf{\small(C)}&$\downarrow $  $\stackrel{k\to\infty}{}$&& $\stackrel{k\to\infty}{}$  $\downarrow$ &\textbf{\small(B)}\\
\\
&$\boxed{\mu^{f_n}}$    &  $\stackrel{n\to\infty}\longrightarrow$ & $\boxed{\mu^{f}}$ &\\
\\
&&\textbf{\small(D)}&&
\end{tabular}
\end{center}
Note that each of the limits should be understood as convergence in the metric given by \eqref{eqn: metric sup BL norm}.\\
\\
At this moment we are not yet able to prove that this scheme represents reality, but the results of this paper yield additional insight. Regarding the limit processes (A), (B), (C) and (D) in the scheme above, the following can be said:
\begin{enumerate}[(A)]
\item For fixed $k\in\N$, there is an obvious \emph{candidate} for $\lim_{n\to\infty} \mu^{f_n,k}$, namely $\mu^{f,k}$, the Euler approximation corresponding to $f$. Note that, for fixed $k$, $\mu^{f,k}$ is well-defined by \eqref{eqn: Euler scheme} and Theorem \ref{thm:exist uniq v prescribed} (that is, \cite[Propositions 3.1 and 3.3]{EvHiMu_JDE}). However, it is nontrivial to actually show that $\sup_{t\in[0,T]}\| \mu^{f_n,k}_t - \mu^{f,k}_t \|^*_\BL\to 0$ as $n\to\infty$.\\
The most straight-forward way to prove this, would be to look at the interval $(t^k_j,t^k_{j+1}]$, estimate $\| \mu^{f_n,k}_\tau - \mu^{f,k}_\tau \|^*_\BL$ for arbitrary $\tau\in(t^k_j,t^k_{j+1}]$, take the supremum over $\tau$ and finally the maximum over $j$. We will now point out what the problem is with this strategy.\\
Like in \eqref{eqn: Euler scheme}, let the semigroup $Q$ denote the operator that maps initial data to the solution in the sense of Definition \ref{def: mild soln fixed v}. From now on, we use $Q^{u,g}$ to denote the semigroup associated to velocity $u\in\BL([0,1])$ and right-hand side $F_g$, where $g$ is piecewise bounded Lipschitz. For any $\tau\in(t^k_j,t^k_{j+1}]$, we have
\begin{align*}
\mu^{f_n,k}_\tau &= Q^{u,f_n}_{\tau-t^k_j}\,\mu^{f_n,k}_{t^k_j},\\
\mu^{f,k}_\tau &= Q^{\bar{u},f}_{\tau-t^k_j}\,\mu^{f,k}_{t^k_j},
\end{align*}
with $u=v[\mu^{f_n,k}_{t^k_j}]$ and $\bar{u}=v[\mu^{f,k}_{t^k_j}]$. To estimate $\| \mu^{f_n,k}_\tau - \mu^{f,k}_\tau \|^*_\BL$ from above, one would use the triangle inequality and obtain three terms of the form
\begin{equation}\label{eqn:typical term triangle est}
\| Q^{w,g}_{\tau-t^k_j}\,\nu - Q^{\bar{w},\bar{g}}_{\tau-t^k_j}\,\bar{\nu} \|^*_\BL,
\end{equation}
for the appropriate choices of 
\begin{equation*}
w,\bar{w}\in\{ v[\mu^{f_n,k}_{t^k_j}],  v[\mu^{f,k}_{t^k_j}]\},   \hspace{0.1\linewidth}  g,\bar{g}\in\{ f_n, f \},   \hspace{0.1\linewidth}  \nu,\bar{\nu}\in\{ \mu^{f_n,k}_{t^k_j}, \mu^{f,k}_{t^k_j}  \}.
\end{equation*}
Specifically, we would use the triangle inequality in such a way that in each term on the right-hand side two of these three pairs of variables are the same (e.g.~$w=\bar{w}$, $g=\bar{g}$ and $\nu\neq\bar{\nu}$). An estimate for each term of the form \eqref{eqn:typical term triangle est} follows e.g.~from \cite[Proposition 3.5]{EvHiMu_JDE}, \cite[Proposition 4.2]{EvHiMu_JDE}, \cite[Corollary 2.9]{EvHiMu_SIMA} or \cite[Lemma 2.10]{EvHiMu_SIMA}. There is freedom in how exactly we apply the triangle inequality, and thus in which specific terms of the form \eqref{eqn:typical term triangle est} appear on the right-hand side. However, each of these approaches results in an upper bound that depends on the Lipschitz constant of $f_n$ (mostly via the bounded Lipschitz norm of $f_n$), which is unbounded as $n\to\infty$ if $f$ is discontinuous.\\
For the moment, it is therefore still an open question whether $\mu^{f_n,k}$ converges to $\mu^{f,k}$ as $n\to\infty$.

\item We emphasize that the results of Theorems \ref{thm:mild is weak v=v(x)} and \ref{thm:mild is weak v=v(mu)} in this paper, do not hinge on the assumption that $f\in\BL([0,1])$. Theorem \ref{thm:mild is weak v=v(x)} is stated explicitly to hold for any $f$ that is piecewise bounded Lipschitz. Theorem \ref{thm:mild is weak v=v(mu)} is only restricted to $f\in\BL([0,1])$, because it builds on Theorem \ref{thm: exist uniq nonlin}, which we managed to prove only for continuous $f$ in \cite[Theorem 3.10]{EvHiMu_SIMA}. Let us \textit{assume} that Theorem \ref{thm: exist uniq nonlin} does provide the convergence of Euler approximations \emph{even} for $f$ that is piecewise bounded Lipschitz. Note that Theorem \ref{thm:exist uniq v prescribed} demands that the discontinuities of $f$ and zeroes of $v$ do not coincide; for the sake of the argument here, ignore this complication and assume that we \emph{can} generalize Theorem \ref{thm: exist uniq nonlin} to piecewise bounded Lipschitz functions $f$. 
In that case the statement of Theorem \ref{thm:mild is weak v=v(mu)} still holds: the resulting mild solution is a weak solution. The current proof of Theorem \ref{thm:mild is weak v=v(mu)} requires a slight modification, though, since in \eqref{eqn:est conv Ff pair int with f BL} we used that $f\in\BL([0,1])$ to obtain immediately an estimate against $\|f\|_\BL$.\\
The convergence in \eqref{eqn:conv pair Ff}, can however be obtained for $f$ piecewise bounded Lipschitz, using arguments very much like the ones leading to \eqref{eqn:conv pair int}. These arguments involve the portmanteau theorem 
(in a slightly more general form than Lemma \ref{lem:Portemanteau}) and the dominated convergence theorem. See Appendix \ref{app:conv pair f not BL} for more details.

\item Since $n$ is fixed in this step and since $f_n\in\BL([0,1])$ for each $n$, this convergence result is covered by \cite[Theorem 3.10]{EvHiMu_SIMA}. In the current work, we show that the limit $\mu^{f_n}$ is a weak solution; see Theorem \ref{thm:mild is weak v=v(mu)}.
\item We \textit{assumed} above that $\mu^f$ can be obtained as the limit of the Euler approximations $\mu^{f,k}$, and we stress here that this is only an assumption. When trying to relate the mild solutions $\mu^{f_n}$ to the mild solution $\mu^f$ by letting $n$ tend to infinity, one encounters the following problem: a mild solution is defined as the limit of a sequence of Euler approximations, but this does not provide a useful characterization of the limit itself. We suggested in \cite{EvHiMu_SIMA} to use the weak formulation of the problem as an alternative characterization. In the current paper we show that $\mu^{f_n}$ is a weak solution. A next step would be to show that this solution converges in some sense (e.g.~weakly) as $n\to\infty$.
\end{enumerate}
Even if the addressed problems in steps (A)--(D) would be resolved, an additional argument is needed to conclude that the two limits $k\to\infty$ and $n\to\infty$ commute -- and thus the scheme above is fully correct.\\
Under the aforementioned assumption that $\mu^{f,k}\to\mu^f$ as $k\to\infty$, the mild solution obtained via the route (A)-(B) is a weak solution, as argued above. On the other hand, our considerations regarding route (C)-(D) lead to an alternative $\mu^f$, obtained as the (weak?) limit of the weak solutions $\mu^{f_n}$. It might be possible to show (easily) that this $\lim_{n\to\infty}\mu^{f_n}$ is also a weak solution.\\
\\
Finally, assume that we want to compare $\lim_{k\to\infty}\mu^{f,k}$, resulting from (A)-(B), to \linebreak $\lim_{n\to\infty}\mu^{f_n}$, resulting from (C)-(D), based on the fact that they are both weak solutions. We did not prove in this paper that weak solutions are unique. To be able to identify $\lim_{k\to\infty}\mu^{f,k}$ with $\lim_{n\to\infty}\mu^{f_n}$, therefore an extra uniqueness criterion or selection criterion might be needed. The issue of uniqueness is nontrivial and it probably plays a role which specific weak formulation is used and which space of test functions is chosen. This topic is `work in progress' and will (hopefully) be the subject of a follow-up paper. 

\section*{Acknowledgements}
I thank Sander C.~Hille (Leiden University, The Netherlands) for fruitful discussions and his valuable suggestions, and I thank the anonymous reviewer, whose comments helped to improve the manuscript.\\
Until 2015, I was a member of the Centre for Analysis, Scientific computing and Applications (CASA), and the Institute for Complex Molecular Systems (ICMS) at Eindhoven University of Technology, The Netherlands, supported financially by the Netherlands Organisation for Scientific Research (NWO), Graduate Programme 2010. Some of the results presented here, were obtained during my time in Eindhoven.

\begin{appendix}
\section{Basics of measure theory}\label{sec: basics meas th}
We denote by $\CM([0,1])$ the space of finite Borel measures on the interval $[0,1]$ and by $\CM^+([0,1])$ the convex cone of positive measures included in it.
For $x\in [0,1]$, $\delta_x$ denotes the Dirac measure at $x$. Let
\begin{equation}\label{pairing}
\pair{\mu}{\phi} :=\int_{[0,1]} \phi \,d\mu
\end{equation}
denote the natural pairing between measures $\mu\in\CM([0,1])$ and bounded measurable functions $\phi$. The {\em push-forward} or {\em image measure} of $\mu$ under Borel measurable $\Phi:[0,1]\to [0,1]$ is the measure $\Phi\#\mu$ defined on Borel sets $E\subset [0,1]$ by
\begin{equation}
(\Phi\# \mu)(E) := \mu\bigl(\Phi^{-1}(E)\bigr).\label{eqn: def push-forw chapter bc}
\end{equation}
One easily verifies that $\langle\Phi\# \mu,\phi\rangle=\langle\mu,\phi\circ\Phi\rangle$.\\
\\
The {\em total variation norm} $\|\cdot\|_\TV$ on $\CM([0,1])$ is defined by
\begin{equation*}
\|\mu\|_{\TV}:= \sup\left\{\pair{\mu}{\phi}\; : \; \phi\in C_b([0,1]),\ \|\phi\|_\infty\leqslant 1 \right\},
\end{equation*}
where $C_b([0,1])$ is the Banach space of real-valued bounded continuous functions on $[0,1]$ equipped with the supremum norm $\|\cdot\|_\infty$. It follows immediately that for $\Phi:[0,1]\to [0,1]$ continuous, $\|\Phi\#\mu\|_\TV\leqs\|\mu\|_\TV$.\\
\\
In \cite{EvHiMu_JDE, EvHiMu_SIMA} and in this paper, we mainly use a different norm on $\CM([0,1])$ that we will introduce now. Let $\BL([0,1])$ be the vector space of real-valued bounded Lipschitz functions on $[0,1]$, equipped with the norm
\begin{equation}\label{def:BL-norm}
\|\phi\|_\BL := \|\phi\|_\infty + |\phi|_\Lip
\end{equation}
for which this space is a Banach space \cite{Fortet-Mourier:1953,Dudley1}. Here,
\begin{equation*}
|\phi|_\Lip  := \sup\left\{ \frac{|\phi(x)-\phi(y)|}{|x-y|}\; : \; x,y\in [0,1],\ x\neq y\right\}
\end{equation*}
is the Lipschitz constant of an arbitrary $\phi\in\BL([0,1])$. With this norm $\BL([0,1])$ is a Banach algebra for pointwise product of functions:
\begin{equation}\label{eqn:Banach algebra}
\|\phi\cdot\psi\|_\BL \le \|\phi\|_\BL\,\|\psi\|_\BL.
\end{equation}
Let $\|\cdot\|_\BL^*$ be the dual norm of $\|\cdot\|_\BL$ on the dual space $\BL([0,1])^*$, i.e.~for any $x^*\in\BL([0,1])^*$ its norm is given by
\begin{equation*}
\|x^*\|_{\BL}^*:= \sup\left\{ |\pair{x^*}{\phi}|\; :\; \phi\in\BL([0,1]),\ \|\phi\|_\BL\leqslant1\right\}.
\end{equation*}
A linear {\em embedding} of $\CM([0,1])$ into $\BL([0,1])^*$ is provided by the map $\mu\mapsto I_\mu$ with $I_\mu(\phi):=\pair{\mu}{\phi}$; see \cite[Lemma~6]{Dudley1}. Thus $\|\cdot\|^*_\BL$ induces a norm on $\CM([0,1])$, which is denoted by the same symbols. It is called the dual bounded Lipschitz norm or Dudley norm. Generally, $\|\mu\|_\BL^*\leqslant\|\mu\|_\TV$ for all $\mu\in\CM([0,1])$. For positive measures the two norms coincide:
\begin{equation}\label{TV norm is dual BL norm for pos measures}
\|\mu\|_{\BL}^*=\mu([0,1])=\|\mu\|_{\TV} \hspace{1 cm}\text{for all }\mu\in\CM^+([0,1]).
\end{equation}
In general, the space $\CM([0,1])$ is not complete for $\|\cdot\|_\BL^*$. We denote by $\CMc([0,1])_\BL$ its completion, viewed as closure of $\CM([0,1])$ within $\BL([0,1])^*$. The space $\CM^+([0,1])$ is complete for $\|\cdot\|^*_\BL$, hence closed in $\CM([0,1])$ and $\CMc([0,1])_\BL$.\\
\\
The $\|\cdot\|_\BL^*$-norm is convenient also for integration. In \cite[Appendix~C]{EvHiMu_JDE} some technical results about integration of measure-valued maps were collected.

\section{Proof of convergence statement \eqref{eqn:conv pair Ff} for discontinuous $f$}\label{app:conv pair f not BL}
The proof is based on the dominated convergence theorem and makes use of the portmanteau theorem.\\
\\
For each $\nu\in\CM([0,1])$,
\begin{equation}\label{eqn:est Ff TV for f not BL}
\|F_f(\nu)\|_\TV \leqs \|f\|_\infty\cdot \|\nu\|_{\TV}
\end{equation}
holds, even for $f$ piecewise bounded Lipschitz, hence not necessarily in $C_b([0,1])$. This inequality was previously used in the proof of \cite[Proposition 3.1]{EvHiMu_JDE}; a proof can be found in \cite[Lemma 4.3.1]{Evers_PhD}. The proof makes use of an approximation of $f$ by $C_b$-functions.\\ 
Note that each mild solution $(\nu_t)_{t\in[0,T]}$ corresponding to a fixed $v\in\BL([0,1])$ and with $f$ piecewise bounded Lipschitz satisfies 
\begin{equation}
\|\nu_t\|_\TV\leqs\|\nu_0\|_\TV \exp(\|f\|_\infty t),\label{eqn:TV bounded mild soln f not BL}
\end{equation} 
according to \cite[Proposition 3.3]{EvHiMu_JDE}; the proof of that claim was left to the reader there. The argument is based on Gronwall's Lemma applied to the estimate
\begin{equation}
\|\nu_t\|_\TV \leqs \|\nu_0\|_\TV + \int_0^t \|f\|_\infty \cdot \|\nu_s\|_\TV \, ds,\label{eqn:VOC TV est before Gronwall}
\end{equation}
while the latter inequality follows from the variation of constants formula and the estimate $\|P_\tau\mu\|_\TV\leqs \|\mu\|_\TV$ that holds for any $\mu\in\CM([0,1])$, due to \eqref{eqn:TV norm Pt}. 
Moreover, \eqref{eqn:est Ff TV for f not BL} is used to obtain \eqref{eqn:VOC TV est before Gronwall}.\\
\\
The arguments in \cite[Lemmas 3.4 and 2.8(i)]{EvHiMu_SIMA} build on \eqref{eqn:TV bounded mild soln f not BL}, and thus the result of Lemma \ref{lem: set timeslices bdd} in this paper is still valid if $f$ is piecewise bounded Lipschitz. That is, the uniform estimate 
\begin{equation}
\|\mu^k_t\|_\TV\leqs\|\nu_0\|_\TV \exp(\|f\|_\infty T)\label{eqn:bdd timeslices f not BL}
\end{equation}
holds for all $k$ and $t$.\\
\\
For each $t\in[0,T]$ fixed, note that $\psi(\cdot,t)$ is in $C_b([0,1])$ and thus,
\begin{equation}
\left|\pair{F_f(\mu^k_t)}{\psi(\cdot,t)}\right| \leqs \sup_{\tau\in[0,T]}\|\psi(\cdot,\tau)\|_\infty \cdot \|F_f(\mu^k_t)\|_\TV \leqs \sup_{\tau\in[0,T]}\|\psi(\cdot,\tau)\|_\infty \cdot \|f\|_\infty\cdot \|\mu^k_t\|_{\TV},\label{eqn:est pair Ff t TV}
\end{equation}
for each $k\in\N$, where the last step is due to \eqref{eqn:TV bounded mild soln f not BL}. Combining \eqref{eqn:est pair Ff t TV}   with \eqref{eqn:bdd timeslices f not BL}, we obtain that the function $t\mapsto\pair{F_f(\mu^k_t)}{\psi(\cdot,t)}$ is bounded uniformly in $k$ and for each $t\in[0,T]$ by the (constant) function $t\mapsto \sup_{\tau\in[0,T]}\|\psi(\cdot,\tau)\|_\infty \cdot \|f\|_\infty\cdot \|\nu_0\|_\TV \exp(\|f\|_\infty T)$.\\
\\
Next we prove that $t\mapsto\pair{F_f(\mu^k_t)}{\psi(\cdot,t)}$ converges pointwise. We emphasize that in the rest of this proof we work here under the (unproved) hypothesis that the Euler approximations $\mu^k$ converge to some unique mild solution $\mu$ even if $f$ is piecewise bounded Lipschitz.\\
For each $t\in[0,T]$, the function $x\mapsto f(x)\cdot\psi(x,t)$ is bounded and measurable, because $f$ is piecewise bounded Lipschitz, and $\psi(\cdot,t)\in C^1_b([0,1])$. Trivially, $\pair{F_f(\mu^k_t)}{\psi(\cdot,t)}=\pair{\mu^k_t}{f\cdot\psi(\cdot,t)}$ holds.\\
By hypothesis $\mu^k\to\mu$ in the metric \eqref{eqn: metric sup BL norm}, hence we know in particular that $\pair{\mu^k_t}{\phi}\to\pair{\mu_t}{\phi}$ as $k\to\infty$ for all $\phi\in\BL([0,1])$. Instead of the simple version of the portmanteau theorem presented here in Lemma \ref{lem:Portemanteau}, we use \cite[Theorem 13.16]{Klenke}, which states that convergence against $\BL([0,1])$ is equivalent to convergence against bounded measurable functions that are discontinuous only on a nullset. One such function is $x\mapsto f(x)\cdot\psi(x,t)$, since $f$ is assumed to have finitely many discontinuities. Hence, $\pair{\mu^k_t}{f\cdot\psi(\cdot,t)}\to\pair{\mu_t}{f\cdot\psi(\cdot,t)}$ as $k\to\infty$ for all $t\in[0,T]$, or equivalently
\begin{equation}
\pair{F_f(\mu^k_t)}{\psi(\cdot,t)}\to\pair{F_f(\mu_t)}{\psi(\cdot,t)}, \text{ as } k\to\infty.\label{eqn:pw conv pair Ff f not BL}
\end{equation}
Due to the uniform bound on $t\mapsto\pair{F_f(\mu^k_t)}{\psi(\cdot,t)}$ and the pointwise convergence \eqref{eqn:pw conv pair Ff f not BL}, the dominated convergence theorem yields in particular that \eqref{eqn:conv pair Ff} holds even for $f$ that is \emph{piecewise} bounded Lipschitz:
\begin{equation}
\int_0^T \pair{F_f(\mu^k_t)}{\psi(\cdot,t)}\,dt \stackrel{k\to\infty}{\longrightarrow} \int_0^T \pair{F_f(\mu_t)}{\psi(\cdot,t)}\,dt.
\end{equation}

\end{appendix}

\bibliographystyle{plain}
\bibliography{refs}

\end{document}